\documentclass[12pt]{amsart}
\usepackage[T1]{fontenc}
\usepackage[latin9]{inputenc}
\usepackage[letterpaper]{geometry}
\usepackage{units}
\usepackage{amsthm}
\usepackage{amssymb}

\numberwithin{equation}{section} 
\numberwithin{figure}{section} 
\theoremstyle{plain}
  \theoremstyle{plain}
  \newtheorem*{thm*}{Theorem}
  \theoremstyle{plain}
  \newtheorem*{lem*}{Lemma} 
 \theoremstyle{definition}
 \newtheorem*{defn*}{Definition}
  \theoremstyle{remark}
  \newtheorem*{rem*}{Remark}

\theoremstyle{plain}
\newtheorem{thm}{Theorem}[section]
  \theoremstyle{plain}
  \newtheorem{cor}[thm]{Corollary}
  \theoremstyle{plain}
  \newtheorem{lem}[thm]{Lemma}
  \theoremstyle{remark}
  \newtheorem*{claim*}{Claim}
  \theoremstyle{definition}
  \newtheorem*{problem*}{Problem}
  \theoremstyle{plain}
  \newtheorem{prop}[thm]{Proposition}

  \theoremstyle{plain} 
  \newtheorem{ex}[thm]{Example}

  \theoremstyle{plain}
  \newtheorem*{question*}{Question}

  \newcounter{casectr}

\newcommand{\N}{\mathbb N}
\newcommand{\Z}{\mathbb Z}

\begin{document}

\title{Applications of $p$-deficiency and $p$-largeness}

\author{J.\,O.\,Button and A.\,Thillaisundaram}
\maketitle

\section*{Introduction}

At the turn of the twentieth century, William Burnside posed the question:
can a finitely generated group which is torsion (meaning that
every element has finite order) be infinite? Initially many
thought the answer would be negative. Then in 1964, Golod and Shafarevich
\cite{key-Gol} gave such an example. Afterwards other constructions
were obtained by Adjan \&
Novikov \cite{key-AN}, Olshanski\u{\i} \cite{key-Ol}, Grigorchuk
\cite{key-GriGp}, and Gupta \& Sidki \cite{key-GS}, to name a few.
Very recently Schlage-Puchta \cite{key-19} gave a remarkably
straightforward proof of the existence of such groups by introducing
the concept of $p$-deficiency, which we review in Section 2.

Now little work has been done on what we call here the related
Burnside problem (see \cite{key-K}, Problem 8.52): do there exist
infinite finitely \emph{presented} torsion groups? Certainly, no conclusive
answer has been established.

We contribute to the related Burnside problem by proving that there
do not exist infinite finitely presented
torsion groups with $p$-deficiency greater than one. 
This follows from our main result:\\ 
For a group $G$ and a prime $p$, we say that
\begin{itemize}
\item $G$ is \emph{large} (as introduced in \cite{prd}) 
if some (without loss of generality normal) subgroup with
finite index admits a non-abelian free quotient; 
\item $G$ is \emph{$p$-large} (as introduced in \cite{key-Lac})
if some normal subgroup with index a power
of $p$ admits a non-abelian free quotient.
\end{itemize}
\begin{thm*}
For a finitely presented group $G$ with $p$-deficiency greater than
one, $G$ is $p$-large.
\end{thm*}

A corollary of this result is that the finitely generated infinite
$p$-groups constructed by Schlage-Puchta are never finitely presented
and do not have property (T).

Note that our main result runs parallel to the famous Baumslag-Pride
Theorem \cite{key-BS}: groups with at least two more generators than
relators (that is deficiency greater than one) are large. In this
proof, the authors show that for a group $G$ of deficiency greater
than one, when $n$ is sufficiently large, $G$ has an index $n$
normal subgroup $\overline{H}$ that surjects onto $F_{2}$. As a
consequence of the proof of this theorem
(taking $n$ to be a suitably large power of $p$), we have that groups
of deficiency greater than one are also $p$-large for all primes
$p$. Naturally, our interest lies in groups of deficiency at most one.

Further to the related Burnside problem, we present here various
applications of our main result. A significant application is to
Coxeter groups where all labels are powers of an odd prime $p$,
following up on work by Grigorchuk \cite{key-Gri} which considered
the case $p=2$.

We then consider presentations with $p$-deficiency greater than one
where the number of relators can be finite or infinite. We compare
these with Golod-Shafarevich presentations in \cite{key-Gol} and show
that for all primes at least 7, any presentation with $p$-deficiency
greater than one (or even equal to one as long as the usual 
deficiency is not equal to one) is Golod-Shafarevich. Moreover, although
this is not true in the case of the three exceptional primes, 
we show that there exists a finite index subgroup which is
Golod-Shafarevich.

In the last section
we show that groups having presentations with $p$-deficiency greater
than one are non amenable and discuss related results on having infinite
quotients that either have property (T) or which are amenable.  

This paper is organised as follows. We begin with a small section
that consists of some terminology and notation, as well as relevant
facts on $p$-groups. Next, we present our main result. Thirdly we give
a brief outline on consequences of $p$-largeness and then the
application to Coxeter groups is illustrated. The subsequent section
combines our concept of $p$-deficiency with Golod-Shafarevich groups.
Lastly, we consider amenability and property (T).
$\ $

\section{Some basics}

Throughout this paper, $p$ denotes a prime number, $\mathbb{N}$
denotes the set of natural numbers $\{1,2,\ldots\}$, and $\mathbb{F}_{p}$
is the field of $p$ elements.
A group $G$
is a \emph{$p$-group} if every element of $G$ has order a $p$th
power. (Note: $G$ can be finite or infinite.)

For a group $G$, with $X$ a set of elements in $G$, $\langle X\rangle$
denotes the subgroup of $G$ generated by $X$, 
and $\langle\langle X\rangle\rangle$
denotes the normal closure in $G$ of the subset $X$ of
$G$. We denote the free group on a set $X$ by $F(X)$ and the free group
of rank $r$ by $F_r$. For $x,y\in G$,
we write $[x,y]$ to mean $xyx^{-1}y^{-1}$.

A \emph{presentation} $\langle X|R\rangle$ for a group $G$ is a
set $X$, and $R$ a subset of $F(X)$, such that 
$G\cong\nicefrac{F(X)}{\langle\langle R\rangle\rangle}$.
The elements $x\in X$ are called \emph{generators} and $r\in R$
are called \emph{relators}. A group $G$ is 
\begin{itemize}
\item \emph{finitely generated} if there exists some presentation 
$\langle X|R\rangle$ of $G$ with $|X|<\infty$;
\item \emph{finitely presented} if there exists some presentation 
$\langle X|R\rangle$
of $G$ with both $|X|$ and $|R|$ finite.
\end{itemize}
For $H$ a proper subgroup of $G$, we say that $H$ is \emph{subnormal}
in $G$ if there exists a positive integer $n$ and a normal series
in $G$ such that\[
H=H_{n}\trianglelefteq H_{n-1}\trianglelefteq\ldots
\trianglelefteq H_{1}\trianglelefteq H_{0}=G.\]

For a finite $p$-group $P$, every finite index subgroup is subnormal.
In fact,
we can say more:
\begin{thm*}
\cite{key-Nil} Suppose that $P$ is a finite $p$-group.

a) Every normal subgroup $N\trianglelefteq P$ may be included into
some central series of $P$ with factors of order $p$.

b) Every subgroup $H\le P$ may be included into some subnormal series
of $P$ with factors of order $p$.\end{thm*}

In an infinite group $G$ we have that any subnormal subgroup $H$ of
prime power index in $G$ contains a subgroup $N$ which is of prime power
index and normal in $G$ (for instance by considering the derived $p$-series
of $G$).

We note that if $G$ surjects onto a direct product
$C_{p}\times\ldots\times C_{p}$, then $G$ factors through 
$\nicefrac{G}{G^{p}G'}$.
In addition, we note that in a finite $p$-group $P$, $P^{p}P'$
is the Frattini subgroup, $\Phi(P)$, of $P$.
By definition, the Frattini subgroup of an arbitrary group is the
intersection of all maximal subgroups. Also the Frattini subgroup
is a characteristic subgroup and it consists of all non-generators
of the group: $x\in G$ is a \emph{non-generator} if whenever 
$x\in X\subseteq G$
and $\langle X\rangle=G$ then $\langle X\backslash\{x\}\rangle=G$.
$\ $

We finish this section with a very brief mention of properties (T),
($\tau$), and amenability. A finitely generated group $G$ has
\emph{property} (T) if every isometric action of $G$ on a Hilbert
space has a global fixed point. It has \emph{property} ($\tau$) if for
some (equivalently any) finite generating set $S$ for $G$, the set
of Cayley graphs $\mbox{Cay}(G/N,S)$ form an expander family, where $N$ varies
over all finite index normal subgroups of $G$. It is 
\emph{amenable} if there
exists a finitely additive, left invariant, probability measure. All
three properties are preserved under quotients, extensions, and
subgroups of finite index, whereas amenability is further preserved under
arbitrary subgroups. Property (T) implies ($\tau$), but not vice versa.
However an amenable group with (T) must be finite, as must a residually
finite, amenable group with ($\tau$).

\section{Groups with $p$-deficiency greater than one}

We now recall Schlage-Puchta's beautiful and straightforward construction
of infinite finitely generated $p$-groups as given in \cite{key-19}.

\begin{defn*}
Let $G$ be a finitely generated group and $\langle X|R\rangle$ a
presentation for $G$. Throughout this paper we always assume that
$X$ is finite, but we allow $R$ to be finite or infinite, in which case
we refer to $\langle X|R\rangle$ as a finite or infinite presentation
respectively.
The \emph{deficiency} of $G$ with presentation
$\langle X|R\rangle$ is\[
\text{def}(G;X,R)=|X|-|R|.\]
If $|R|=\infty$, then we define $\text{def}(G;X,R)$ to be $-\infty$.
\end{defn*}

If we take such a
presentation $G=\langle x_1,\ldots ,x_d|w_1,w_2,\ldots\rangle$ 
where the $w_j$ are elements of the free group $F_d$ 
then the Reidemeister-Schreier rewriting process (see \cite{key-Lyn})
will produce a presentation
of an index $i$ subgroup $H$ of $G$, for which there will
be $(d-1)i+1$ generators.
As for relators, each $w_j$ gives rise to $i$ relators in this presentation for
$H$, and these $i$ relators are
conjugate to each other in $F_d$. In particular if this presentation is finite,
whereupon $r$ relators for $G$ results in $ir$ for $H$,
we see that $\mbox{def}-1$ is multiplicative, 
that is for $H\leq_f G$ with index
$i$ we have a presentation $H\cong\langle Y|S\rangle$ with
$\mbox{def}(H;Y,S)-1=i(\mbox{def}(G;X,R)-1)$.

Now suppose that $G$ has a normal subgroup $H$ of index $p$ 
for some prime $p$, so that $H$ can be thought
of as the kernel of a homomorphism $\theta$ from $G$ onto $C_p$. We can 
assume on changing the presentation
that $\theta(x_1)=\ldots =\theta(x_{d-1})=0$ but $x_d$, which we 
henceforth rename
$t$, maps to 1. This is because not every generator is sent to zero, so
by reordering them we can assume that $x_d$ does not.
Moreover we can take $\theta(t)=1$ without changing the kernel.
Next we replace each of the first $d-1$ generators $x_i$ with 
$x_it^{-n_i}$ where $\theta(x_i)=n_i$ and we rewrite the new relators
in terms of the new generators. In particular the number of generators
and relators is unchanged.
The reason for changing the presentation in this way is that the 
Reidemeister-Schreier process now takes on a particularly simple
form. The generators of $H$ are $s:=t^p$ and $x_{i,j}:=t^jx_it^{-j}$
for $1\leq i\leq d-1$ and $0\leq j\leq p-1$, and the relators are
$t^jw_kt^{-j}$ for $1 \leq k\leq r$ (or for all $k\ge 1$ if $R$ is infinite)
and $0\leq j\leq p-1$ but written in
terms of the above generating set for $H$.   

Schlage-Puchta's construction rests on two easily verifiable but ingenious
observations. The first comes into play when one of the relators for $G$, 
let us say $w_1$ when written in terms of the new generators for $G$,
is a $p$th power of some other word $v\in F_d$. Let our homomorphism
from $G$ to $C_p$ now have domain $F_d$ (but we will still
name it $\theta$) by composing $\theta$ with the
natural map from $F_d$ to $G$. If $\theta(v)=0$ then it is clear
that rewriting $t^iw_1t^{-i}=(t^ivt^{-i})^p$ in terms of our generators
for $H$ yields the same result as rewriting $t^ivt^{-i}$ in terms of these
generators and raising to the power $p$. In particular one $p$th power
relator in the presentation for $G$ results in $p$ $p$th power relators
in that for $H$. This cannot be true if 
$\theta(v)\neq 0$ because then $t^ivt^{-i}$ is not in the kernel of
$\theta$ and so cannot be rewritten in terms of these generators. But
it can seen that the
$(t^ivt^{-i})^p$ are conjugate in $F_{(d-1)p+1}$ when written in terms
of the $(d-1)p+1$ generators for $H$, because on increasing $i$
by one during rewriting, a cyclic permutation of the previous rewritten
word is produced. As conjugates of a given relator can be ignored,
in this situation we can say that a $p$th power
has produced a single relator in the presentation for $H$, not 
$p$ separate relators. We refer to the above process as
\emph{Puchta rewriting}.

The second crucial observation is to quantify the above in such a way
that it can work for presentations with infinitely many relators.  
\begin{defn*}\cite{key-19}
For a prime $p$, the \emph{$p$-deficiency} of $G$ with presentation
$\langle X|R\rangle$ is\[
\text{def}_{p}(G;X,R)=|X|-\sum_{r\in R}p^{-\nu_{p}(r)},\]
where $\nu_{p}(r)=\max\left\{ k\ |\ \exists 
w\in F(X),\ w^{p^{k}}=r\right\} \ge0$. \\
If $|R|=\infty$ and the sum does not converge then we define
$\mbox{def}_p(G;X,R)$ to be $-\infty$, although we can certainly have
cases where $|R|=\infty$ but $\mbox{def}_p(G;X,R)$ is finite.

Note: it is always the case that $\text{def}_{p}(G;X,R)\ge\text{def}(G;X,R)$
and if $\nu_{p}(r)=0$ for all $r\in R$, then 
$\text{def}_{p}(G;X,R)=\text{def}(G;X,R)$.
\end{defn*}

From the above we see that 
Schlage-Puchta has proved a multiplicative property for
$\mbox{def}_p-1$ similar to that for $\mbox{def}-1$. 
That is, if $H\le_{p^{k}}G$ and $H$ is subnormal
in $G$ then there exists $Y$ and $S$ with $H\cong\langle Y|S\rangle$ such that 
$\text{def}_{p}(H;Y,S)-1=p^{k}\left(\text{def}_{p}(G;X,R)-1\right)$.

For ease of notation, we shall in future write $\text{def}_{p}(G)$
and $\text{def}(G)$ instead of $\text{def}_{p}(G;X,R)$ and $\text{def}(G;X,R)$
respectively as we will always have a specific presentation in mind, even
though these quantities can vary drastically over all presentations defining
$G$.
\begin{defn*}
The {\it $p$-rank} $\text{d}_{p}(G)$ of a finitely generated group $G$ 
is the dimension of the homology group $H_{1}(G;\mathbb{F}_{p})$.
Equivalently it is the rank (here meaning the number
of generators) of $\nicefrac{G}{G'G^{p}}$ and can be thought of as
the maximum number of copies of $C_{p}$ onto which
$G$ surjects: $G\twoheadrightarrow C_{p}\times\ldots\times C_{p}$.
\end{defn*}

For a presentation $G=\langle X|R\rangle$ with $d=|X|$ but where $R$ is either
infinite or finite (with ranges of sums adjusted accordingly),
suppose exactly $s$ of the relators in $R$ are
not powers of $p$ - that is, $y\in R$ cannot be expressed as $v^{p}$
for some $v\in F(X)$. If $\mbox{def}_p(G)$ is finite then 
$s$ must be finite too.
Our presentation for $G$ thus can be written in the form \[
\langle x_{1},\ldots,x_{d}|w_1,\ldots ,w_s,
w_{s+1}^{p^{b_{s+1}}},w_{s+2}^{p^{b{}_{s+2}}},\ldots\rangle,\]
for $w_{i}\in F(X)$ where $b_{s+1},b_{s+2},\ldots\ge 1$.
Referring to our definition of $\text{d}_{p}(G)$, it is natural to
consider $\nicefrac{G}{G'G^{p}}$ in our next step. 
In additive notation, assuming all generators commute, we have

\[
\frac{G}{G'G^{p}}=\langle x_{1},\ldots,x_{d}\ |\ pw=0\ \forall w\in F(X),
\ w_{1},\ldots,w_{s}\rangle\]
where, for $1\le i\le s$, we can express $w_{i}=a_{i1}x_{1}+\ldots+a_{id}x_{d}$
with $a_{ij}\in\mathbb{F}_{p}$ for $j=1,\ldots,d$. As $G$ is finitely
generated, $\nicefrac{G}{G'G^{p}}$ is a finite abelian group of exponent
$p$ and so can be viewed as a vector space
over $\mathbb{F}_{p}$, with $d$ variables (the generators) and $s$
not necessarily linearly independent equations (the relators). Thus,
as a vector space, $\text{dim}\left(\nicefrac{G}{G'G^{p}}\right)\ge d-s$.
So we establish the inequality\[
\text{d}_{p}(G)\ge d-s.\]
Now, by definition, $\text{def}_{p}(G)=|X|-s-\sum_{i=s+1}^{\infty}p^{-b_{i}}
\le d-s$. From the above, we make
the key deduction:\begin{equation} \label{eq:KeyD}
\text{d}_{p}(G)\ge\text{def}_{p}(G).\end{equation}
This has an important consequence.
\begin{cor}\label{coinf} 
If there exists a presentation for the finitely
generated group $G$ with $\mbox{def}_p(G)\geq 1$ then $G$ is infinite.
\end{cor}
\begin{proof}
From \eqref{eq:KeyD} we see that $\mbox{d}_p(G)\geq 1$, so that there exists
$H\unlhd G$ with index $p$. But multiplicity of $p$-deficiency implies
that $\mbox{def}_p(H)\ge 1$ so we can iterate indefinitely.
\end{proof}

Now in order to create finitely generated infinite $p$-groups,
Schlage-Puchta merely takes a finite generating set $X$ with $|X|\geq 2$
and an enumeration $w_1,w_2,\ldots$ of all elements in $F(X)-\{\mbox{id}\}$.
On forming the presentation $G=\langle X|w_1^{p^{n_1}}, w_2^{p^{n_2}},\ldots
\rangle$, we see that $G$ is clearly a $p$-group and is infinite if
$\sum_{i=1}^{\infty}1/p^{n_i}\le |X|-1$ by Corollary \ref{coinf}.
He also shows that $G$ is
non amenable if the sum is less than $|X|-1$ using rank gradient, which
will be discussed in Section 6.

The motivation for our main result in this section
is the related Burnside problem. As $p$-deficiency provides such a natural
way of constructing infinite finitely generated torsion groups, an obvious
question to ask is whether it can be adapted to produce examples which are
finitely presented (see Corollary \ref{nfp} below). 
However if $G$ is large then it certainly is not an infinite torsion group.
$\ $

The following is the goal of this section.
\begin{thm}
\label{thm:Main} (Main Result) For a finitely presented group $G$
with $\text{def}_{p}(G)>1$, $G$ is $p$-large.\end{thm}

\begin{proof}
The following is a vital tool that we will need.

$\ $

\begin{thm}
\label{1.15} (\cite{key-Lac} Theorem 1.15)
Let $G$ be a finitely presented group,
and let $p$ be a prime. Then the following are equivalent.

(1) $G$ is $p$-large;

(2) $G$ has an abelian $p$-series with rapid descent.
\end{thm}

$\ $

We use Schlage-Puchta's multiplicity result to show that the
group $G$ of Theorem \ref{thm:Main} has an abelian $p$-series of
rapid descent. Then we apply Theorem \ref{1.15}.

$\ $

\begin{defn*}
\cite{key-Lac} An \emph{abelian $p$-series }for a group $G$ is
a sequence of finite index subgroups \[
G=G_{1}\trianglerighteq G_{2}\trianglerighteq\ldots\]
such that $\nicefrac{G_{i}}{G_{i+1}}$ is an elementary abelian $p$-group
for each natural number $i$.

An abelian $p$-series $\left\{ G_{i}\right\} $ has \emph{rapid descent}
if\[
\inf_{i}\dfrac{\text{d}_{p}(\nicefrac{G_{i}}{G_{i+1}})}
{\left[G:G_{i}\right]}>0.\]

\end{defn*}

\begin{rem*}
The equivalence of Theorem \ref{1.15} enables us to conclude that
for a $p$-large group $H$ that is subnormal of $p$th power index
in a supergroup $G$, then $G$ is $p$-large. This is apparent from
the fact that an abelian $p$-series of $H$ with rapid descent can
be extended to an abelian $p$-series of $G$, still with rapid descent.
We also give a direct proof in Lemma \ref{splg}.
\end{rem*}
$\ $ 

Having laid out all the machinery that is required, we construct the
required abelian $p$-series \[
G=G_{0}\trianglerighteq G_{1}\trianglerighteq G_{2}\trianglerighteq\ldots\]
as follows. Henceforth, $G\cong\langle X|R\rangle$ is a finitely
presented group with $\text{def}_{p}(G)>1$, as in the hypothesis
of Theorem \ref{thm:Main}.

We assume that $|X|\ge2$, as if $|X|=1$ then our group $G$ is cyclic
and we cannot possibly have $\text{def}_{p}(G)>1$.
Moreover we know from \eqref{eq:KeyD} that $\text{d}_{p}(G)>1$, 
which means a surjection $G\twoheadrightarrow C_{p}$ exists. 
Therefore, $\exists G_{1}\trianglelefteq_{p}G$.

$\ $

We set $G_{0}:=G$. As $\text{def}_{p}(G_{0})-1=\varepsilon>0$,
Puchta rewriting means that
\[\text{def}_{p}(G_{1})-1=p\left(\text{def}_{p}
\left(G_{0}\right)-1\right)=p\varepsilon.\]
 If $p\varepsilon\ge1$, then we advance to the next paragraph. Else,
proceed down the series to get\[
G_{n}\trianglelefteq_{p}G_{n-1}\trianglelefteq_{p}\ldots
\trianglelefteq_{p}G_{1}\trianglelefteq_{p}G_{0}\]
with $\text{def}_{p}(G_{n})-1=p^{n}\varepsilon\ge1$. The formation
of such a series is possible since 
$\text{d}_{p}(G_{i})\ge\text{def}_{p}(G_{i})>1$ for each $G_i$.
We have $\left[G:G_{n}\right]=p^{n}$.

$\ $

A new abelian $p$-series is initiated from $H:=G_{n}$ (or $G_{1}$
if $p\varepsilon\ge1$). We show instead that $H$ is $p$-large,
which implies that $G$ is $p$-large.

Let $H_{1}\trianglelefteq_{p}H$ (which exists by a similar argument
as to $G_{1}\trianglelefteq_{p}G$). 
Then $\text{def}_{p}(H_{1})-1=p\left(\text{def}_{p}(H)-1\right)\ge p$. 

So $\text{d}_{p}(H_{1})\ge p+1$, but `$\text{d}_{p}(H_{1})\ge p$'
suffices for our proof, and hence 
$H_{1}\twoheadrightarrow\underbrace{C_{p}\times\ldots\times C_{p}}$.

$\qquad\qquad\qquad\qquad\qquad\qquad\qquad\qquad
\qquad\qquad\qquad\qquad\qquad\qquad\qquad\qquad\qquad\quad p\text{ times}$

Set $H_{2}:=\ker\left(H_{1}\twoheadrightarrow(C_{p})^{p}\right)$.
Then $\nicefrac{H_{1}}{H_{2}}\cong\underbrace{C_{p}\times\ldots\times C_{p}}$
, so $|\nicefrac{H_{1}}{H_{2}}|=p^{p}$, and 
 $\text{d}_{p}(\nicefrac{H_{1}}{H_{2}})=p$.

$\qquad\qquad\qquad\qquad\qquad\qquad\qquad\qquad\qquad\qquad\ p\text{ times}$

We compute that

\[
\dfrac{\text{d}_{p}(\nicefrac{H_{1}}{H_{2}})}
{\left[H:H_{1}\right]}=\dfrac{p}{p}=1.\]
Following the same line of thought,

\[
\text{def}_{p}(H_{2})-1=p^{p}\left(\text{def}_{p}
(H_{1})-1\right)\ge p^{p+1}\ \Longrightarrow\ \text{d}_{p}(H_{2})\ge p^{p+1}.\]
Similarly, $H_{3}:=\ker\left(H_{2}\twoheadrightarrow
\left(C_{p}\right)^{p^{p+1}}\right)$.
So $|\nicefrac{H_{2}}{H_{3}}|=p^{p^{p+1}}$, and hence 
$\text{d}_{p}(\nicefrac{H_{2}}{H_{3}})=p^{p+1}$.

$\ $

The next fraction gives us the same value.

\[
\frac{\text{d}_{p}(\nicefrac{H_{2}}{H_{3}})}
{[H:H_{2}]}=\frac{p^{p+1}}{p^{p+1}}=1\]
Proceeding in this manner, we have ensured that\[
\inf_{i}\dfrac{\text{d}_{p}(\nicefrac{H_{i}}{H_{i+1}})}
{\left[H:H_{i}\right]}=1>0.\]
\end{proof}

$\ $

As the reader might have noted, a recursive formula for $\text{d}_{p}
(\nicefrac{H_{i}}{H_{i+1}})(=\left[H:H_{i}\right])$
is\[
\text{d}_{p}(\nicefrac{H_{i}}{H_{i+1}})=p^{\text{d}_{p}\left
(\nicefrac{H_{i-1}}{H_{i}}\right)}\cdot\text{d}_{p}(
\nicefrac{H_{i-1}}{H_{i}})\]
\[
=p^{\left[H:H_{i-1}\right]}\cdot\left[H:H_{i-1}\right].\]
$\ $ 

It is worth pointing out that this brings to light a whole new collection
of large groups with finite negative deficiency. For example,
on taking one's favourite prime $p$,
the 2-generator group\[
\langle x,y|w_{1}^{p},\ldots,w_{p-1}^{p}\rangle\]
is $p$-large, regardless of which words
$w_{1},\ldots,w_{p-1}\in F(x,y)$ are chosen.

$\ $
\begin{rem*} The proof of Theorem \ref{thm:Main} works if the given
presentation for $G$ is either finite or infinite. However, in the latter
case, suppose the presentation for $G$ is $\langle x_1,\ldots
,x_d|r_1,r_2,\ldots\rangle$. By \cite{nbh} Corollary 12 of the classic
1937 paper by B.\,H.\,Neumann, if $G$ has some finite presentation
then there exists $l$ with $G=\langle x_1,\ldots ,x_d|r_1,\ldots ,
r_l\rangle$. But as this is a truncation of a presentation with
$p$-deficiency greater than 1, we obtain a finite presentation for $G$
with $p$-deficiency greater than 1 anyway. This also applies if we have an
infinite presentation for $G$ which has $p$-deficiency exactly 1.
\end{rem*}
We have two corollaries, the first of which is immediate from
Theorem \ref{thm:Main} and this remark.
\begin{cor}\label{nfp}
There do not exist infinite finitely presented torsion groups which,
for some prime $p$, have a presentation with $p$-deficiency greater
than 1 or an infinite presentation with $p$-deficiency equal to 1.
\end{cor} 

In particular the infinite finitely generated $p$-groups constructed by
Schlage-Puchta are definitely not finitely presented.

Also we can show that these groups do not have property (T).

\begin{cor} \label{pnot}
A group possessing a finite presentation with $p$-deficiency greater
than 1, or an infinite presentation with $p$-deficiency at least
1, does not have property (T).
\end{cor}
\begin{proof} We may assume the presentation is infinite.
By a result \cite{sha} of Shalom, if a finitely generated group has property
(T) then only finitely many relations suffice to confirm this. But as in
the previous corollary, such a finite presentation defines a large group
which cannot have (T).
\end{proof}

\begin{rem*}   
Given an infinite presentation
$G=\langle x_1,\ldots,x_d|r_1,r_2,\ldots\rangle$ as above, define
$G_l=\langle x_1,\ldots,x_d|$  $r_1,r_2,\ldots ,r_l\rangle$. Suppose that
P is a group theoretic property that is preserved by prequotients. We can
then ask: if $G_l$ has P for all $l$ then does $G$ have P? We know that
this is true for being infinite and not having (T). It also holds
for being non abelian, non nilpotent, and having first Betti number or
$p$-rank at least $k$. But by the above it is false for any property
implied by largeness but not held by torsion groups. This includes, for
instance, having virtual first Betti number at least $k$, or containing a
non abelian free group or an element of infinite order.

Two other examples of this phenomenon are in \cite{gdlhlb}
where it is shown that
the groups formed by truncating the standard presentation of the
Grigorchuk group are all large, as well as \cite{bmbk} Chapter IV Theorem 7
which does
the same for the restricted wreath product $\Z\wr\Z$. In particular,
the above is false for the property of being non amenable, or even
being non soluble.
\end{rem*}

\section{Properties of $p$-largeness}

Largeness was introduced in \cite{prd} as an important property of
finitely generated groups which is invariant under finite index subgroups
and supergroups, as well as prequotients. We now look at the appropriate
properties of $p$-largeness. First it is clear that $p$-largeness is
preserved by prequotients because the index and normality of a
subgroup are preserved under the inverse image of a homomorphism.
However a perfect large group, such as $A_5*A_5$, cannot be $p$-large
for any $p$ because of the following. Here we write $H\leq_f G$ for
$H$ a finite index subgroup of $G$ and $N\unlhd_{p^k}G$ when $N$ is a
normal subgroup with index a power of $p$.
\begin{lem} If the $p$-rank $d_p(G)\leq 1$ then $G$ is not $p$-large.
\label{frat}\end{lem}
\begin{proof}
Suppose $P$ is a finite $p$-quotient of $G$. Then the $p$-rank $d_p(P)$
is also at most 1 so $P/P'P^p$ is cyclic. But by the Lemma in Section 1
on the Frattini subgroup of a finite $p$-group,
this means $P$ is cyclic. However if $N\unlhd_{p^k} G$ with $N$
surjecting onto the free group $F_2$ of rank 2 then $N'N^p$ is
characteristic in $N$, hence normal in $G$, and $N/N'N^p$ surjects onto
$C_p\times C_p$. Thus $G/N'N^p$ is a non cyclic finite $p$-group.
\end{proof}
Also we cannot remove normality from the definition of $p$-large.
\begin{ex}\label{examp}
\end{ex}
If $G$ has a non (sub)normal subgroup of index a power of $p$ that 
surjects onto
$F_2$ then $G$ need not be $p$-large. For instance let $G=\Z*C_2$. 
As this has $p$-rank 1 unless $p=2$, $G$ cannot be
$p$-large for primes $p\geq 3$ by Lemma \ref{frat}, although it is 
$2$-large and
hence large. However $G$ has a non normal subgroup of index 3 which is
$F_2*C_2$.
\begin{lem} If $G$ is $p$-large and $H\leq_f G$ then $H$ is $p$-large.
\end{lem}
\begin{proof} Suppose $N\unlhd_{p^k} G$ with $N$ surjecting onto $F_2$.
Then $H\cap N$ is normal and of finite index in both $N$ and $H$, 
so $H\cap N$ 
surjects onto a non abelian free group by restricting the homomorphism
from $N$. Moreover the index $[H:H\cap N]=[HN:N]$ and as $N$ normal implies
that $HN$ is a subgroup of $G$, we have $[HN:N]$ divides $p^k$.
\end{proof}

Now let us consider supergroups of finite index, so suppose that
$H$ is $p$-large and $H\leq_f G$. Then $G$ need not be $p$-large,
even if the index is $p$, as in Example \ref{examp}. Also it is not true
even if $H$ is normal in $G$ and the index is coprime to $p$, or divides
$p$, for instance $C_2*C_2*C_2$ is not $p$-large for $p>2$ but $G$ has
free non abelian
normal subgroups of index $2n$ for all $n\in\N$. But the one case where
we can transfer $p$-largeness to a finite index supergroup is when $H$ is
normal in $G$ and has index a power of $p$. Indeed subnormality works here:
\begin{lem} \label{splg}
If $N$ is $p$-large and is a subnormal subgroup of $G$ with
index $p^k$ then $G$ is $p$-large.
\end{lem}
\begin{proof}
There exists $M\unlhd_{p^l} N$ with $M$ surjecting onto $F_2$. Although
$M$ need not be normal in $G$, it is subnormal with index $p^{k+l}$.
Therefore we have a subgroup $L\unlhd_{p^m}G$ with $L\leq M$, as
mentioned in Section 1, and $L$ surjects onto $F_2$ as well.
\end{proof}

An important consequence of a finitely generated group being large is that
it has many finite quotients. We have a variant on this which is that if
$G$ is $p$-large then it has many $p$-groups as quotients, where we include
both finite and infinite $p$-groups. Theorem 4 in \cite{miols} states that
if $G$ is large with $N\unlhd_f G$ mapping onto $F_2$ and $p$ is any prime
then $G$ possesses uncountably many residually finite torsion quotients
(here we always count up to isomorphism) such that in each quotient
the image of $N$ is a $p$-group. Moreover Corollary 1 of this paper shows that
if $G$ is $p$-large then $G$ possesses uncountably many residually finite 
quotients that are $p$-groups.

We can give here a shorter proof of these facts provided we assume a priori
that for each prime $p$ there exist uncountably many finitely generated 
residually finite $p$-groups which are just infinite
(for instance, see \cite{grdg}). To do this we will adapt
Proposition 4.5 in \cite{ers1} which shows that if $H\leq_f G$ where $H$ 
maps onto an infinite group with property (T) then so does $G$. In fact
property (T) is not used directly
in the proof, thus it can be replaced by other suitable properties. 
Moreover a very similar approach was earlier used by $\Pi$.\,N.\,Neumann
in \cite{neu} Section 2 
and was also known to J.\,S.\,Wilson in \cite{wld1}.
\begin{prop}\label{pmn}
Suppose that P is a property of finitely generated groups which is preserved
by quotients, finite index subgroups and supergroups, and finite direct
products. If $G$ is finitely generated and has $H\leq_f G$ such that $H$ maps 
onto an infinite group with P then so does $G$.
\end{prop}
\begin{proof} Take $K\unlhd_fG$ with $K\leq H$. Then $K$ also has an infinite
quotient with P by restriction. As P is preserved by quotients, we can assume
that there is $L\unlhd K$ where $K/L$ is just infinite and has P.
Although we need not have $L\unlhd G$,
the normaliser of $L$ in $G$ contains $K$ and so it has finite index in $G$.
Consequently we can assume that $L_1,\ldots ,L_n$ are all of the
conjugates of $L$ in $G$. On putting $N=L_1\cap\ldots \cap L_n$ we have
$N\unlhd G$ and a natural injective homomorphism
\[\pi:K/N\rightarrow \Gamma=K/L_1\times\ldots\times K/L_n\]
which is surjective under each projection onto a factor, so that $K/N$ is a 
subdirect product
of $K/L\cong K/L_i$. Moreover if we let $\Pi=\pi(K/N)\cong K/N$ and
$\Pi_i=\Pi\cap(K/L_i)$, we have that $\Pi_i$ is normal in $K/L_i$.

We can assume that $\Pi_i$ has finite index in the just infinite group
$K/L_i$, because otherwise
$\Pi_i$ is trivial in which case the map from $\Gamma$ which forgets the
$i$th component is injective when restricted to $\Pi$. Now $\Pi$ contains
$\Pi_1\times\ldots\times\Pi_n$ which has finite index in $\Gamma$, so
$\Gamma$ having P and $\Pi\leq_f\Gamma$ means that $\Pi$ does too.
Thus $G/N$ has the finite index subgroup $K/N$ with P, so $G/N$ has P and
is infinite.
\end{proof}

Thus if a finitely generated group $G$ has $H\le_f G$ where $H$ has an
infinite quotient with (T), or ($\tau$), or which is amenable then so does
$G$. In fact the argument can be simplified in the amenable case: if the
property P in Proposition \ref{pmn} is preserved by all subgroups, finite
index supergroups and finite direct products, but not necessarily by
quotients, then we do not need $K/L$ to be just infinite in the proof
because if $K/L$ has P then so will $K/N$, as it is a subgroup of $\Gamma$.

In the original argument by $\Pi$.\,M.\,Neumann, the specific property
considered is that of containing a given infinite simple group $S$.
However the features required of this property for the proof to work
are that it is preserved by finite index subgroups, all supergroups, and
(in place of the just infinite quotient) a maximal condition which says
if $K$ has a quotient $K/L_0$ with this property then there exists
$L_0\le L\unlhd K$ such that $K/L$ and any non trivial normal subgroup
of $K/L$ also have this property. We then proceed exactly as in the
first paragraph of Proposition \ref{pmn} and conclude that either
$\Pi_i$ has our property, so $\Pi$ and hence $G/N$ does, or $\Pi_i$ is
trivial but this cannot happen for all $i$.

\begin{cor}\label{impq}
If $G$ is large with $N\unlhd_f G$ mapping onto $F_2$ and $p$ is any prime
then $G$ possesses uncountably many residually finite torsion quotients
such that in each quotient
the image of $N$ is a $p$-group. Moreover if $G$ is $p$-large then $G$ 
possesses uncountably many residually finite quotients that are $p$-groups.  
\end{cor}
\begin{proof}
As there exist uncountably many finitely generated residually finite
$p$-groups which are just infinite,
there must exist an integer $n$ such that uncountably many
of these are generated by $n$ elements and so are quotients of $F_n$.
If $K\unlhd_f G$ and $K$ maps
onto any non-abelian free group $F$ then we can assume $K$ maps onto $F_n$ by
taking $C$ characteristic and of finite index in $K$ such that $C$
surjects onto a free group of rank at least $n$. 
For instance we can take $C=K'K^q$ where $q$ is a power
of $p$, which is characteristic in $K$ and will map onto $F'F^q$. By taking
a suitably big power, $F'F^q$ will have arbitrarily high index in $F$ and
hence arbitrarily high rank.
We can then replace $C$ by $K$, ensuring that $K$ has uncountably many
residually finite just infinite $p$-groups.

Now if $P$ is the property of being a torsion group then the conditions for
Proposition \ref{pmn} are satisfied. On following the proof, we see 
that if $K$ surjects onto the just infinite 
$p$-group $K/L$ then $G$ surjects onto the
infinite torsion group $G/N$, with $K/N$ a subdirect product of $K/L$, thus
$K/N$ is an infinite $p$-group.

Being residually finite does not satisfy the appropriate properties but,
as noted in \cite{wld1}, if $K/L$ is residually finite then so is $G/N$:
On taking $x\notin N$ (but which we can assume is in $K$)
then, as $N=\cap gLg^{-1}$ over all $g\in G$, there
exists $g\in G$ with $g^{-1}xg\notin L$. So we can take a finite index
normal subgroup $M/L$ of $K/L$ that misses $g^{-1}xgL$. Thus $x$ is not
contained in the finite index subgroup $gMg^{-1}$ of $G$ but $N$ is.

So we have an uncountable set $\{L_i:i\in I\}$ with $L_i\unlhd K$ where, by
using the above construction, each $L_i$ gives rise
to $N_i\unlhd G$ such that $G/N_i$ is an
infinite residually finite torsion group. Now if
only countably many $N_i$ (up to isomorphism of $G/N_i$) occur
then we must have a particular $N_{i_0}$ which is obtained from
uncountably many $L_i$. But then $K/N_{i_0}$ is a finite index subgroup
of $(K/L_i)^{n_i}$ for each of these $L_i$ and some $n_i\in\N$.
Now there are uncountably many
isomorphism classes when we vary over all 
$(K/L_i)^{n_i}$ because each one contains the
finitely generated subgroup $K/L_i$ and a finitely generated group has
only countably many finitely generated subgroups. But the finitely
generated group $K/N_{i_0}$ can only sit with finite index
in countably many supergroups (up to isomorphism). 

Finally if $K\unlhd_{p^k}G$ then we can take $C$ as before, so that 
on replacing $C$ with $K$ and using the above, we now have $K\unlhd_{p^l}G$
with $K/N$ and $G/K$ both $p$-groups, so $G/N$ is too.
\end{proof}
  
\section{Application to Coxeter groups}

In this section, we extend a result of Grigorchuk \cite{key-Gri}
to odd primes.
\begin{defn*}
\cite{key-Gri} A \emph{Coxeter group} is a group with presentation,\[
C=\langle x_{1},\ldots,x_{n}|x_{1}^{2},\ldots,x_{n}^{2},
(x_{i}x_{j})^{m_{ij}},1\le i<j\le n\rangle,\]
where $m_{ij}$ is an integer at least 2, or $\infty$ (the case $m_{ij}=\infty$
means that the relation $(x_{i}x_{j})^{m_{ij}}$ is not present).
\end{defn*}
$\ $ 

We also note a dichotomy of Coxeter groups, which is listed under
Theorem 14.1.2 of \cite{key-Cox}.
\begin{thm}
\cite{key-Cox} A Coxeter group is either virtually abelian or large.
\end{thm}
$\ $

Grigorchuk's result is the following.
\begin{thm}
\label{thm:Gri}\cite{key-Gri} Every Coxeter group, that is not virtually
abelian and for which all labels $m_{ij}$ (referring to the presentation
above) are powers of 2 or infinity, surjects onto uncountably many
infinite residually finite 2-groups.
\end{thm}
In order to extend Grigorchuk's result to odd primes, it is necessary
to alter our focus to subgroups of Coxeter groups. This is because,
for $n\in\mathbb{N}$ and $p$ an odd prime, Coxeter groups of the
form\[
G=\langle x_{1},\ldots,x_{n}|x_{1}^{2},\ldots,x_{n}^{2},
(x_{i}x_{j})^{m_{ij}},1\le i<j\le n,\ \&\ m_{ij}\ \text{is a }p\text{th power 
(including }\infty)\rangle,\]
do not admit surjections onto $p$-groups. To see why this is the
case, we note that all generators of $G$ are of order two, and
so they must be mapped to the identity in a $p$-group (since $p$
is an \emph{odd} prime, there are no elements of order two in a $p$-group).

Therefore we consider instead the index two normal subgroup $P$ of
$G$, which has presentation,\[
P=\langle a_{1},\ldots,a_{n-1}|a_{1}^{m_{12}},\ldots,a_{n-1}^{m_{1n}},
(a_{i-1}a_{j-1}^{-1})^{m_{ij}}\rangle\]
where $2\le i<j\le n$ and $m_{rs}$, $1\le r<s\le n$, is as given
in $G$. This normal subgroup $P$ has the geometric interpretation
as the {}``orientation-preserving'' subgroup of $G$. The above
presentation for $P$ is obtained via a straightforward application
of the Reidemeister-Schreier rewriting process (details in Subsection
4.1).

Before we state our result, we give a name to these 
{}``orientation-preserving'' subgroups of Coxeter groups where the labels
are all powers of $p$.
\begin{defn*}
For $p$ an odd prime and $n\ge2$, suppose that $C$ is a Coxeter group with
all labels $m_{ij}$ a power of $p$ (or $\infty$).
A \emph{$p$-Coxeter subgroup} is the index 2 subgroup of $C$ given by the 
kernel of $\theta:C\rightarrow C_2$ where $C(x_i)=1$ for all $i$.
It has the form\[
\langle a_{1},\ldots,a_{n-1}|a_{1}^{m_{12}},\ldots,
a_{n-1}^{m_{1n}},(a_{i-1}a_{j-1}^{-1})^{m_{ij}}\rangle\]
where $2\le i<j\le n$ and $m_{rs}$, $1\le r<s\le n$ are $p$th
powers (or infinity).
\end{defn*}
$\ $

Here is our result of this section, and we embark on its proof in
what follows.
\begin{thm}
\label{thm:p-Cox}Every $p$-Coxeter subgroup that is not virtually
abelian surjects onto uncountably many infinite residually finite
$p$-groups.
\end{thm}
$\ $

It is insufficient to establish that the
relevant $p$-Coxeter subgroups are \emph{large} for the result to
hold, but \emph{$p$-largeness} is enough by Corollary \ref{impq}.
It so happens that all $p$-Coxeter
subgroups that are not virtually abelian, and hence large, do satisfy
the theorem. The virtually abelian $p$-Coxeter subgroups are all
the cyclic ones and $\langle a_{1},a_{2}|a_{1}^{3},a_{2}^{3},
(a_{1}a_{2}^{-1})^{3}\rangle$,
as we shall see from the proof of Theorem \ref{thm:p-Cox}. For this
last group, $\langle a_{1},a_{2}|a_{1}^{3},a_{2}^{3},
(a_{1}a_{2}^{-1})^{3}\rangle$,
which is isomorphic to 
$T=\langle a_{1},a_{2}|a_{1}^{3},a_{2}^{3},(a_{1}a_{2})^{3}\rangle$,
there is a useful geometric interpretation. Its supergroup \[
G=\langle x_{1},x_{2},x_{3}|x_{1}^{2},x_{2}^{2},x_{3}^{2},
(x_{1}x_{2})^{3},(x_{1}x_{3})^{3},(x_{2}x_{3})^{3}\rangle,\]
can be viewed as the group generated by reflections in the sides of
an equilateral triangle. The group $G$ has $\mathbb{Z}\times\mathbb{Z}$
as an index six subgroup (the subgroup of translations), and likewise
$T$ has $\mathbb{Z}\times\mathbb{Z}$ as an index three subgroup.

\begin{proof}
For convenience, we will refer to the statement {}``there exists
uncountably many surjections to infinite residually finite $p$-groups''
as $(\dagger)$, and we denote \[
S_{n}(p):=\langle a_{1},\ldots,a_{n-1}|a_{1}^{p},\ldots,a_{n-1}^{p},
\left(a_{i-1}a_{j-1}^{-1}\right)^{p},\ 2\le i<j\le n\rangle.\]
We have $n\geq 3$ or else $S_n(p)$ is cyclic.
For $p>3$, the $p$-deficiency of the following $p$-Coxeter
subgroup\[
S_{3}(p)=\langle a_{1},a_{2}|a_{1}^{p},a_{2}^{p},(a_{1}a_{2}^{-1})^{p}\rangle\]
is $2-\frac{3}{p}>1$. So $S_{3}(p)$ is $p$-large by Theorem \ref{thm:Main} 
and so it satisfies $(\dagger)$.

$\ $

Next, we note the following family of surjections, as was done similarly
in Grigorchuk's paper \cite{key-Gri}. Here the indices are understood
to range as in the definition of $p$-Coxeter subgroups.

\begin{equation}
\langle a_{1},\ldots,a_{n-1}|a_{1}^{p^{k_{12}}},
\ldots,a_{n-1}^{p^{k_{1n}}},\left(a_{i-1}a_{j-1}^{-1}\right)^{p^{k_{ij}}}
\rangle\ \twoheadrightarrow\ \langle a_{1},\ldots,
a_{n-1}|a_{1}^{p},\ldots,a_{n-1}^{p},\left(a_{i-1}a_{j-1}^{-1}\right)^{p}
\rangle\label{eq:1}\end{equation}
where $k_{rs}\in\mathbb{N}$ for $1\le r<s\le n$; and

\begin{equation}
\langle a_{1},\ldots,a_{n}|a_{1}^{p},\ldots,a_{n}^{p},
\left(a_{i-1}a_{j-1}^{-1}\right)^{p}\rangle\ \twoheadrightarrow\ 
\langle a_{1},\ldots,a_{n-1}|a_{1}^{p},\ldots,a_{n-1}^{p},
\left(a_{i-1}a_{j-1}^{-1}\right)^{p}\rangle.\label{eq:2}\end{equation}

$\ $

Let $p>3$. Since $S_{n+1}(p)\twoheadrightarrow S_{n}(p)$ by the
surjection (\ref{eq:2}) and $S_{3}(p)$ satisfies $(\dagger)$, by
composition of surjections of type (\ref{eq:2}), we deduce that $S_{n}(p)$
satisfies $(\dagger)$ for all $n\ge3$. From the family of surjections
(\ref{eq:1}), we conclude that all $p$-Coxeter subgroups with $p>3$
satisfy $(\dagger)$. 

For $n=3$, all such 3-Coxeter subgroups, apart
from $\langle a_{1},a_{2}|a_{1}^{3},a_{2}^{3},(a_{1}a_{2}^{-1})^{3}\rangle$,
have 3-deficiency bigger than 1.
For the last remaining collection of 3-Coxeter subgroups with $n>3$,
we have proven that $S_{4}(3)$ is 3-large with the aid of MAGMA (details
in Subsection 4.2). So as before, $(\dagger)$ holds for all 3-Coxeter
subgroups with $n>3$.
\end{proof}

$\ $

\subsection{Reidemeister-Schreier rewriting for $p$-Coxeter subgroups}

$\ $

For every Coxeter group \[
C=\langle x_{1},\ldots,x_{n}|x_{1}^{2},\ldots,x_{n}^{2},
(x_{i}x_{j})^{m_{ij}}\rangle\]
 where $1\le i<j\le n$, $m_{ij}\in\{2,3,\ldots\}\cup\{\infty\}$, there
exists an index two normal subgroup $P$ which has presentation

\[
P=\langle a_{1},\ldots,a_{n-1}|a_{1}^{m_{12}},\ldots,a_{n-1}^{m_{1n}},
\left(a_{i-1}a_{j-1}^{-1}\right)^{m_{ij}}\rangle\]
where $2\le i<j\le n$ and $m_{rs},\ 1\le r<s\le n$, is as given
in $C$. 

$\ $

The above presentation for $P$ is obtained via the Reidemeister-Schreier
rewriting method: we define $P$ as the kernel of the homomorphism
$G\twoheadrightarrow\{0,1\}$ given by considering the parity of word
lengths. Our Schreier transversal is taken to be the set $T=\{e,x_{1}\}$.

With this transversal, we obtain generators for $P$ of the form 
$tx(\overline{tx})^{-1}$,
where $t\in T$ and $x\in\left\{ x_{1},\ldots,x_{n}\right\} $ and
$\overline{tx}$ is the coset representative of the element $tx$.

For $t=e$: 

$\qquad x=x_{1}:\quad x_{1}x_{1}^{-1}=e$

$\qquad x=x_{2}:\quad x_{2}x_{1}^{-1}=x_{2}x_{1}=:a_{1}$ (note: $x_{1}$
is of order 2)

$\qquad\qquad\qquad\qquad\vdots$

$\qquad x=x_{n}:\quad x_{n}x_{1}^{-1}=x_{n}x_{1}=:a_{n-1}$

$\ $

For $t=x_{1}:$

$\qquad x=x_{1}:\quad x_{1}^{2}=e$

$\qquad x=x_{2}:\quad x_{1}x_{2}=a_{1}^{-1}$ (this equality is obtained
using the relators $x_{1}^{2},\ldots,x_{n}^{2}$)

$\qquad\qquad\qquad\qquad\vdots$

$\qquad x=x_{n}:\quad x_{1}x_{n}=a_{n-1}^{-1}$

$\ $

Now the relators of $P$ are \[
\left(x_{1}x_{2}\right)^{m_{12}}=a_{1}^{m_{12}},\ldots,
\left(x_{1}x_{n}\right)^{m_{1n}}=a_{n-1}^{m_{1n}}\]
and

\[
\left(x_{i}x_{j}\right)^{m_{ij}}=\left(x_{i}x_{1}x_{1}x_{j}
\right)^{m_{ij}}=\left(a_{i-1}a_{j-1}^{-1}\right)^{m_{ij}}\ \text{for }
2\le i<j\le n.\]

$\ $

The other relators of $G$, that is $x_{1}^{2},\ldots,x_{n}^{2}$,
were absorbed above during our construction of the generating set.

$\ $

\subsection{The use of MAGMA}
\begin{claim*}
$P_{0}:=S_{4}(3)=\langle a_{1},a_{2},a_{3}\ |\ a_{1}^{3},a_{2}^{3},
a_{3}^{3},(a_{1}a_{2})^{3},(a_{1}a_{3})^{3},(a_{2}^{-1}a_{3})^{3}\rangle$
is 3-large. (Note: we have replaced generators $a_{2},a_{3}$ as in
the definition of a $p$-Coxeter subgroup, by $a_{2}^{-1},a_{3}^{-1}$
for convenience).\end{claim*}
\begin{proof}
Using MAGMA's LowIndexNormalSubgroups function, we considered the
following index three normal subgroup of $P_{0}$:\[
P_{1}=\langle x,y,z,w\ |\ y^{3},w^{3},[x,z],(w^{-1}y^{-1})^{3},
w^{-1}y^{-1}z^{-1}wx^{-1}yzx\rangle,\]
which was thirteenth on the list of fourteen normal subgroups with
index at most three in $P_{0}$. The above presentation for $P_{1}$
was obtained using MAGMA's Simplify function.

Using MAGMA's LowIndexNormalSubgroups function a further time, we
applied MAGMA's Simplify command to the first index three normal subgroup
of $P_{1}$ on MAGMA's list:\[
P_{2}=\langle a,b,c,d,e,f\ |\ b^{3},e^{3},[c,a^{-1}],(b^{-1}e^{-1})^{3},\]
\[
fe^{-1}a^{-1}bcaf^{-1}a^{-1}c^{-1}b^{-1}ae,f^{-1}a^{-1}bcafc^{-1}
d^{-1}b^{-1}e^{-1}de,afbdcbd^{-1}b^{-1}f^{-1}a^{-1}c^{-1}b^{-1}\rangle.\]
Then we formed the quotient \[
\nicefrac{P_{2}}{\langle\langle a,b,c,e\rangle\rangle}\]
which is isomorphic to $\langle d,f\rangle\cong F_{2}$. Hence $P_{2}$
is 3-large by definition, and since $P_{2}$ is subnormal in $P_{0}$
of index $3^{2}$, we have proved that $P_{0}=S_{4}(3)$ is 3-large,
as required.
\end{proof}
$\ $

\subsection{More $p$-large groups}

$\ $

At the end of \cite{key-Gri}, Grigorchuk mentions the $p$-groups
($p\ge3$ is a prime) of Gupta-Sidki \cite{key-GS} that are 2-generated,
residually finite, branch, and just-infinite. Their generators $x,y$
satisfy the relators $x^{p}$,$y^{p}$, $(x^{i}y^{j})^{p^2}$, $1\le i,j\le p-1$,
as well as infinitely others. We look at the group\[
G=\langle x,y|x^{p},y^{p},(x^{i}y^{j})^{p^{2}}\rangle\]
which surject onto these Gupta-Sidki $p$-groups.

Note that, as\[
(x^{i}y^{j})^{-1}=y^{-j}x^{-i}=y^{p-j}x^{p-i}=
y^{p-j}(x^{p-i}y^{p-j})y^{-(p-j)},\]
we only need the relators $x^{p}$,$y^{p}$, and $(x^{i}y^{j})^{p^{2}}$
with $1\le i\le (p-1)/2$, $1\le j\le p-1$,
so $\text{def}_{p}(G)=2-\frac{2}{p}-\frac{(p-1)^{2}}{2p^{2}}$ which
is greater than 1 for $p\ge3$. Thus, our groups are $p$-large.

The paper then puts forward the following question:\\
\hfill\\
{\bf Problem} Let $p\ge 5$ be prime. Does there exist an infinite,
residually finite $p$-group generated by two elements $x,y$ subject to
the relations $x^p=y^p=1$ and $(x^iy^j)^p=1$ for $1\le i,j\le p-1$?\\ 
\hfill\\

By Corollary \ref{impq} there are uncountably many such examples if the
group $G(p)$ given by this finite presentation is $p$-large. Unfortunately
the $p$-deficiency of all these presentations
is now well below 1. However we can confirm that we do have $p$-largeness
for $G(5)$ and $G(7)$. This was done by computation and
we now give brief details: the paper \cite{mecmp} describes an
algorithm in MAGMA which aims to show that, on input
of a finite presentation, the group $G$ so defined is large. The main
theoretical tool that is used comes from \cite{how}. This states that
if $G$ has a homomorphism $\chi$ onto $\Z$ with kernel $K$
and there exists a prime $p$
such that the mod $p$ Alexander polynomial $\Delta(t)$ (with respect to
$\chi$) is zero then the finite index normal
subgroup $KG^n$ of $G$ maps onto
the free product $C_p*C_p*C_p$ for sufficiently large $n$. Now $KG^n$ is
certainly $p$-large because the free product is too so, on taking $n$ to be 
a power of $p$, we have that $G$ is $p$-large by Lemma \ref{splg}.

However $G(p)$ has no homomorphisms onto $\Z$, but we are done if we can
find (sub)normal subgroups of $p$ power index that do and which satisfy
the above condition on some mod $p$ Alexander polynomial.

For $G(5)$ we found that a subnormal subgroup of index 25
with abelianisation $(C_5)^2\times\Z^4$ satisfies this condition. The
computation was instant. For $G(7)$ the same was true for index 49 and
the abelianisation was $(C_7)^{10}\times\Z^6$. Here the computation
took a few hours to find subgroups of this index, but the Alexander
polynomial computation was instant.

$\ $

\section{Puchta groups versus Golod-Shafarevich groups}

For convenience, we define the following.
\begin{defn*}
We call a finitely generated group $G$ a \emph{Puchta group} if $G$
has a presentation $\langle X|R\rangle$, with $|X|<\infty$, such
that there exists a prime $p$ with $\text{def}_{p}(G;X,R)>1$, and we call
$\langle X|R\rangle$ a \emph{Puchta presentation}. 
\end{defn*}
$\ $

Note that $G$ will always have other presentations which are not Puchta.

In this section, we relate Puchta groups to Golod-Shafarevich groups.
Before we define a Golod-Shafarevich
group, we first note that the Zassenhaus $p$-filtration 
$\left\{ G_{n}\right\} $ is given by $G_{1}=G$, $G_{2}=[G,G]G^{p}$, 
$G_{3}=[G_{2},G]G_{\ulcorner\nicefrac{3}{p}\urcorner}^{p}$,
$\ldots$ , $G_{i}=[G_{i-1},G]G_{\ulcorner\nicefrac{i}{p}\urcorner}^{p}$.

\begin{rem*}
The Zassenhaus $p$-filtration (and hence the lower central $p$-series)
intersects in the identity for a free group $F(X)$. 
\end{rem*}
$\ $

\begin{defn*}

(a) Consider a group presentation $\langle X|R\rangle$, where $X$
is finite but $R=\{r_1, r_2,\ldots \}$ could be finite or infinite.
For the free group $F(X)$ and any non identity $w\in F(X)$, the \emph{degree} 
$\mbox{deg}(w)$ is the maximum $i$ such that $w\in F(X)_i$.
The presentation $\langle X|R\rangle$ is said to
satisfy the \emph{Golod-Shafarevich} \emph{condition} with respect
to $p$ if there exists a real number $t\in(0,1)$ such that \[
1-H_{X}(t)+H_{R}(t)<0\text{ where }H_{X}(t)=|X|t\text{ and }
H_{R}(t)=\sum_{i=1}^{\infty}t^{\mbox{deg}(r_i)}.\]
We call $F(t)=1-H_X(t)+H_R(t)$ the 
\emph{Golod-Shafarevich function} (or \emph{polynomial}
if $R$ is finite) for this presentation.

(b) A group $\Gamma$ is called \emph{Golod-Shafarevich}, or \emph{GS} for
short, if it has a 
presentation satisfying the Golod-Shafarevich condition. It is shown in
\cite{gol} that $\Gamma$ is infinite, and a major strengthening of this
due to Zelmanov in \cite{zel} is that the pro-$p$ completion
$\hat{\Gamma}_p$ contains a non abelian free pro-$p$ group.

We also have a stronger condition: for $G$ finitely presented with
$\text{d}_{p}(G)\ge2$, $G$ is a \emph{strongly Golod-Shafarevich
group} if there exists a finite presentation $\langle X|R\rangle$
for $G$ such that\begin{equation}
\text{def}(G;X,R)+\frac{(\text{d}_{p}(G))^{2}}{4}-\text{d}_{p}(G)>0.
\label{eq:GS}\end{equation}
\end{defn*}
That this implies the GS condition can be seen by assuming that all relators
of degree at least two are equal to two, then examining the associated
quadratic function.
\begin{rem*} An important property of degree with respect to the prime
$p$ that will be used is $\mbox{deg}(w^p)=p$ $\mbox{deg}(w)$ for
any non identity $w\in F(X)$.
\end{rem*}
$\ $
\begin{question*}
Are Puchta groups Golod-Shafarevich groups?
\end{question*}
Answer: not always. Consider, for example, $p=2$ and the presentation
$P=\langle x,y|x^{2}\rangle$. Then $\text{def}_2(P)=2-1/2>1$. 
But $F(t)=1-H_{X}(t)+H_{R}(t)$ is $(1-t)^2$. Another example for $p=2$
which does not even touch the $t$-axis is
$P=\langle x,y,z|x^2,y^2,z^2\rangle$ with $\text{def}_2(P)=3/2$ again and
$F(t)=1-3t+3t^2=(1-3t/2)^2+3t^2/4$. Moreover we can show that $P$ cannot
be GS with respect to any presentation $\langle X|R\rangle$. Suppose so
then we must have $|R|\geq |X|$ because $P$ has finite abelianisation.
But if $r$ is the number of relators of degree 1 in $R$
then $d_2(P)=3$ implies
that $r\geq |X|-3$. In particular we can assume that 
$F(t)=1-3t+t^{i_1}+\ldots+t^{i_n}$ where $n=|R|-|X|+3\geq 3$.
We claim that at least three of these $n$ remaining relators have degree
at most 2. This is because the quotient of $F(X)$ by the third subgroup
$F(X)_3$ in the Zassenhaus 2-filtration is a nilpotent group of class 2
with abelianisation $(C_4)^{|X|}$. However $F(X)/F(X)_3$ surjects onto
the equivalent quotient $P/P_3$, and this has abelianisation
$(C_2)^3$. Consequently at least $|X|$
non trivial relations are needed in $F(X)/F(X)_3$ to get $G/G_3$ as a
quotient, but relators of degree 3 or more in $F(X)$ become trivial
relators in $F(X)/F(X)_3$.

We can also turn $P$ into a torsion
group by adding to the set of relators, $R$, the following:

\[
w^{2^{N(w)}}\text{ for }w\in F(x,y,z)\text{ and }N(w)\gg0,\]
maintaining the condition $\text{def}_{2}(P)>1$. Certainly, this
group is still not a Golod-Shafarevich group with respect to any
presentation as the previous argument works here too.

$\ $

We now find all examples of the above form.
\begin{lem}\label{exmp}
Let $p$ be any prime and let $G=\langle x_1,\ldots ,x_k
|w_1^p,\ldots ,w_l^p\rangle$ be a Puchta presentation such that
$w_1,\ldots ,w_l$ all have degree 1 in $F_k$. Then this presentation
is GS unless $p=2$ and $(k,l)=(2,1),(3,3),(4,4),(4,5),(6,6)$; 
$p=3$ and $(k,l)=
(2,2),(3,4),(3,5)$; or $p=5$, $k=2$ and $l$ equal to 3 or 4.
\end{lem}
\begin{proof} In order for the presentation to be Puchta, we need
$p(k-1)>l$ and for it to be GS we must find out whether $F(t)=1-kt+lt^p$
travels below the $t$-axis when $t\in (0,1)$. We have that the only
point $u>0$ where $F'(u)=0$ satisfies $plu^{p-1}=k$. We
can assume here that $u<1$: otherwise $k\geq pl$, and if $k>l+1$ then
$F(1)<0$ anyway, so we are left with $k\leq l+1$ giving $l+1\geq pl\geq 2l$,
giving just $l=1$, $k=p=2$ and $F(t)=(1-t)^2$, or $l=0$ and $k=1$ but
this has $p$-deficiency 1 for all $p$.

Otherwise we look at the sign of $F(u)=1-ku+ku/p$. We have $F(u)\geq 0$
exactly when $lp^p\geq k^p(p-1)^{p-1}$. Thus in order for the
presentation to be Puchta but not GS, we require
$p^{p+1}(k-1)>k^p(p-1)^{p-1}$. Now for $n$ and $k$ integers at least 2, we
have by induction on $n$
that $k^n/(k-1)\geq 2^n$. Thus we need primes $p$ such that
$p^{p+1}>2^p(p-1)^{p-1}$. Again we show by induction that
$2^n(n-1)^{n-1}\geq n^{n+1}$ for $n\ge 7$, with the base case certainly true.
The inductive step is left as an exercise for the reader, with the
hint that $(n+1)^3\leq 2n^3$ holds for such $n$. Thus we have no examples 
unless $p=2,3$ or 5. Indeed for $p\geq 7$
we do not even have examples where the presentation has
$p$-deficiency equal to 1 unless the deficiency is 1 because we do
not even have equality in the above.

As for the three exceptional primes, for $p=2$ we need $k^2-8k+8<0$ leaving
only $k=2,3,4,5,6$. Now for each value of $k$ we see which integers $l$
satisfy both $p(k-1)>l$ and $lp^p\geq k^p(p-1)^{p-1}$, giving us the answer
above. For $p=3$ we need $4k^3-81k+81<0$ which is increasing for
$k\geq 3\sqrt{3}/2$ so we need only check $k=2,3$ for the 2 equations
involving $l$. Finally $p=5$ gives rise to $4^4k^5-5^6k+5^6<0$ and the
left hand side is increasing for at least $k\geq 2$, so we only check
$k=2$ and find just $l=3$ or 4.
\end{proof}

Although the above presentations are of a special form, we can use this
to get general results for Puchta groups.
\begin{thm}\label{nallp}
If $G$ is a finitely generated group with a Puchta presentation for any
prime $p>5$ then $G$ is GS with respect to this presentation.
\end{thm}
\begin{proof}
We first assume we have a finite presentation
\[\langle x_1,\ldots ,x_k|w_1^{p^{r_1}},\ldots ,w_l^{p^{r_l}}\rangle\]
and we consider the associated GS function 
\[f(t)=1-kt+\sum_{i=1}^l t^{\mbox{deg}(w_i^{p^{r_i}})}.\]
But we have $\mbox{deg}(w_i^{p^{r_i}})\geq p^{r_i}$ so
$t^{\mbox{deg}(w_i^{p^{r_i}})}\leq t^{p^{r_i}}$ for $t\in (0,1)$.
Consequently let us define the \emph{Puchta polynomial}
(or \emph{Puchta function} in the case of infinite presentations) to be
$g(t)=1-kt+\sum t^{p^{r_i}}$. Then if we can find $t\in (0,1)$
with $g(t)<0$, we have $f(t)<0$ too.
Moreover if any relator is not a proper $p$th power (and there are at
most $k-2$ of these as the $p$-deficiency of the presentation is greater than
1) then we can decrease both $k$ and $l$ by this amount without changing $g$.
Thus we may assume that $k\geq 2$; indeed this also holds if the
presentation has $p$-deficiency exactly 1 unless it has deficiency 1 too. 
Suppose we now express $g(t)$ as
\[1-kt+s_1t^p+s_2t^{p^2}+\ldots +s_nt^{p^n}\mbox{ where }
0\leq s_i\mbox{ and }0<s_n.\]

We define the $p$-deficiency of $g$ to be 
$k-\sum_{i=1}^n s_i/p^i$. Now the idea is to reduce the powers of $p$
occurring in the terms of $g$ in an attempt to increase the function, but
without changing the $p$-deficiency. 
If $n=1$ then we can apply Lemma \ref{exmp}. 
But for $n>1$ we can replace $s_nt^{p^n}$ in $g(t)$ with the 
term $(s_n/p)t^{p^{n-1}}$ to get $h(t)$
without changing the $p$-deficiency of the function. However we have that
$h(t)\geq g(t)$ when $t^{p^{n-1}}\geq pt^{p^n}$. For any $n\geq 2$
we have $p^{n-1}(p-1)\geq p(p-1)$ so this inequality holds whenever
$1\geq pt^{p(p-1)}$.

We now have a polynomial $h$ with degree at most $p^{n-1}$ and we can further
turn terms of the form $t^{p^{n-1}}$ into those of the form
$t^{p^{n-2}}$, again without changing the $p$-deficiency. 
We can continue doing this
until we now have $g_0(t)=1-kt+prt^p$, where the $p$-deficiency of our
presentation is $k-r>1$, with $g_0(t)\geq g(t)$ on $[0,p^{-1/(p^2-p)}]$.
Now we can increase $r$ to $k-1$ to finish with 
$F(t)=1-kt+p(k-1)t^p\ge g_0(t)$ and the $p$-deficiency of 
the function $F$ is exactly 1. Now Lemma \ref{exmp}, in
particular the comment in the proof about when the $p$-deficiency is
equal to 1, tells us that $F(t)$ 
has to go below the $t$-axis for $t\in (0,1)$. Now we only
know that $f(t)\leq F(t)$ for $t\leq p^{-1/(p^2-p)}=y_p$. But evaluating
$F'(t)$ at $y_p$ gives $-k+p^2(k-1)y_p^{p-1}$. Thus $F'(y_p)>0$ if
$p^2p^{-1/p}>k/(k-1)$ which is true for all $p$ because $k/(k-1)\leq 2$
and $p^{1/p}<2$. This means that the minimum $m$ of $F(t)$, where $F'(m)=0$
and $F(m)<0$, is less than $y_p$ as $F'$ is increasing so we have
$f(m)<0$ as well.

Finally if we have an infinite Puchta presentation
$\langle x_1,\ldots ,x_k|w_1^{p^{r_1}},w_2^{p^{r_2}},\ldots\rangle$, we
run through the above proof for each $n\in\N$ using the presentation defined
by taking the first $n$ relators, from which we obtain the corresponding
functions $f_n(t)$. From this we see that
$f_n(t)\leq F(t)$ for all $n$ and for $t\leq p^{-1/(p^2-p)}$, thus
$f_n(m)\leq F(m)<0$ so for our infinite presentation we have that
the power series $f(m)$ converges and is less than 0 too.
\end{proof}

\begin{rem*}
This proof also works for presentations (finite or infinite) with
$p$-deficiency equal to 1, with the exception of those finite presentations
which also have deficiency 1 as here the Puchta polynomial is $1-t$.

It was shown in \cite{wld1} that the soluble
groups of deficiency 1 are precisely the groups $G_k\cong\langle a,t|
tat^{-1}=a^k\rangle$ for $k\in\Z$. In fact these are also the virtually
soluble groups. This gives rise to the question: what, in addition
to these, are the virtually soluble groups with $p$-deficiency equal to
1? By Theorem \ref{nallp} there are no further
examples for primes at least 7, because the pro-$p$ completion of a
virtually soluble group cannot contain a free pro-$p$ group, as opposed
to a GS group. For $p=2$ and 3 examples do exist, such as the infinite
dihedral group and the $(3,3,3)$ triangle group, but we do not know of
an example for $p=5$.
\end{rem*}
   
However we must still deal with $p=2,3$ and 5 where the above theorem
is not true. We find that even here Puchta groups are very close to
being Golod-Shafarevich.

\begin{thm}\label{rmingp}
Every Puchta group $G$ has a Golod-Shafarevich
subgroup of finite index, which we can take to be normal and
of $p$th power index in $G$.
\end{thm}
\begin{proof}
First let us deal with finitely presented Puchta groups. Suppose
that 
\[G=\langle x_1,\ldots ,x_d| r_1^{p^{m_1}}, \ldots ,r_n^{p^{m_n}}\rangle
\mbox{ with Puchta polynomial } 
f(t)=1-dt+\sum_{i=1}^n t^{p^{m_i}}.\] We use a very similar trick as before,
which is that for $t\in[0,p^{-1/(p-1)}]$ we have $pt^p\leq t$ and
$pt^{p^k}\leq t^{p^{k-1}}$ for all $k\geq 1$. This time we 
``convert'' our terms
of the form $t^{p^i}$ into linear terms which increases the function on the
above range. Suppose that the $p$-deficiency of the presentation is $1+e$
so the excess $p$-deficiency $e$ is strictly positive. We end up with
the function $g(t)=1-(1+e)t$ which is negative for $t>1/(1+e)$.
Consequently we have that $G$ is GS if $p^{-1/(p-1)}>1/(1+e)$.

However $e$ can be arbitrarily small. But we know $G$ has a normal subgroup 
$H$ of index $p$, so on Puchta rewriting we
obtain $p(d-1)+1$ generators for $H$ and the relator $r_i^{p^{m_i}}$
yields either one relator to the power $p^{m_i-1}$ or $p$ relators to
the power $p^{m_i}$. In particular the contribution of $r_i^{p^{m_i}}$
to the Puchta polynomial for $H$ is bounded above by either $t^{p^{m_i-1}}$
or $pt^{p^{m_i}}$ and it is the former which is bigger on our range.
On again converting our terms to create a linear function, we have
an upper bound for this polynomial
of the form $1-(pe+1)t$. Repeating this process $k$ times
which produces a subnormal subgroup $H_k$ in $G$ of index $p^k$, we have
that the Puchta polynomial for $H_k$, and hence the GS polynomial,
is bounded above by $1-(p^ke+1)t$ on
$[0,p^{-1/(p-1)}]$. Thus we conclude that $H_k$ is a GS group once
$p^{-1/(p-1)}>1/(p^ke+1)$ and this will happen for large enough $k$.
Moreover we can assume that $H_k$ contains $N\unlhd G$ with index $p^l$
as mentioned in Section 1, 
and as $l\geq k$ the conclusion holds for $N$ too.

As for the case where $G$ is infinitely presented with its Golod-Shafarevich
function bounded above by 
$f(t)=1-dt+\sum_{i=1}^\infty t^{p^{m_i}}$, for any $N\in\N$ let the tail
$T_N(t)$ of $f$ be $\sum_{i=N+1}^\infty t^{p^{m_i}}$. Given $\epsilon>0$,
we can take $N$ large enough that $T_N(t)\leq\epsilon$ for all
$t\in[0,p^{-1/(p-1)}]$. This is because $t^{p^k}\leq 1/p^k$ for all $k\geq 1$ 
(and for $k=0$ too) on the bigger interval $[0,p^{-1/p}]$ because $n^{1/n}$ is 
decreasing in $n$. Thus here $T_N(t)\leq \sum_{i=N+1}^\infty 1/p^{m_i}$, so 
the right hand side can be taken to be $\epsilon$, and this tends to
zero as the $p$-deficiency is convergent.

If $G$ has $p$-deficiency $1+e$ as before, 
we now take $\epsilon<ep^{-1/(p-1)}/(1-p^{-1/(p-1)})$ so that
$\epsilon/(e+\epsilon)<p^{-1/(p-1)}$. We have that the GS function for
$G$ is bounded above by
\[f(t)=1-dt+\sum_{i=1}^N t^{p^{m_i}}+\sum_{i=N+1}^\infty t^{p^{m_i}}.\]
Let $r=\sum_{i=1}^N 1/p^{m_i}$ so that the $p$-deficiency $1+e=d-r-\epsilon$.
Then if we take a subnormal subgroup $H_k$ of index $p^k$ as before and look at
the contribution of the first $N$
relators of $G$ to the GS function of $H_k$, we have
as before an upper bound of $1-(p^k(d-r-1)+1)t=1-(p^k(e+\epsilon)+1)t$ for
$t\leq p^{-1/(p-1)}$. But in the tail, a relator $r_i^{p^{m_i}}$ of 
the given presentation for $G$
contributes to the GS function of a normal index $p$ subgroup $H$
at most either $t^{p^{m_i-1}}$ or $pt^{p^{m_i}}$ and both of
these are bounded by $1/p^{m_i-1}$ on our region. Consequently if we define
the tail $T_H$ for $H$ as those terms arising from the tail of $G$, we
have $T_H(t)\leq p\epsilon$ for allowable $t$. Hence the contribution of the
tail to the GS function of $H_k$ is at most $p^k\epsilon$, giving an
upper bound to the whole function of $p^k\epsilon+1-(p^k(e+\epsilon)+1)t$.
This is less than zero for $t>(p^k\epsilon+1)/(p^k(e+\epsilon)+1)$ which tends
to $\epsilon/(e+\epsilon)$ as $k\rightarrow\infty$. Thus there are points
below the axis in $[0,p^{-1/(p-1)}]$ for all large $k$, so again we can
assume that $H_k$ is normal in $G$.
\end{proof}  

\begin{rem*} In general the GS function of an infinite presentation satisfying 
the GS inequality need not have radius of convergence 1, for instance
$f(t)=1-dt+4t^2+8t^3+16t^4+\ldots$ converges only for $|t|<1/2$ and this
is GS for $d>6$. But an offshoot of the above proof is that
this is not true for Puchta groups.
\end{rem*}

\begin{cor} If $f$ is the GS function of a Puchta presentation with
infinitely many relators then $f$ has radius of convergence 1.
\end{cor}
\begin{proof} 
We do have $f(1)=\infty$, so we show convergence when $t<1$ and do this for
the Puchta function $g(t)=1-dt+\sum_{i=1}^{\infty}t^{p^{m_i}}$ of the
presentation. We saw above that $t^{p^k}\leq 1/p^k$ for all $k\ge 0$ 
on $[0,p^{-1/p}]$ which implies that $\sum_{i=1}^{\infty}t^{p^{m_i}}
\le \sum_{i=1}^{\infty}1/p^{m_i}$ which converges. But now given
$l\ge 1$, take $N$ such that $m_i\ge l$ for all $i\ge N$. Then for any
$k\ge l$ we have $t^{p^k}\le 1/p^k$ on $[0,p^{-l/p^l}]$ and so
$\sum_{i=N}^{\infty}t^{p^{m_i}}\le \sum_{i=N}^{\infty}1/p^{m_i}$. As
$(p^l)^{-1/{p^l}}$ tends to 1 as $l$ tends to infinity, we have convergence 
on $[0,1)$.
\end{proof}

The power of these results is that they produce a wide range of
Golod-Shafarevich group presentations, both finite and infinite, where
the GS condition is confirmed merely by calculating the $p$-deficiency
rather than having to deal with GS functions. Indeed it is straightforward
to produce for any $p$ a finitely generated $p$-group with an infinite
presentation which is both Puchta and GS. By Corollary \ref{pnot} this
group does not have property (T). Now although GS groups need not have
property (T), for instance non abelian free groups are GS, we do not
know of any explicit \emph{torsion} GS groups in the literature until now
that are proven not to have (T). Recently examples of torsion GS groups with
(T) were found and this will be mentioned further in the next section. 

We finish with a result on finitely presented Puchta groups.
\begin{thm}
Every finitely presented Puchta group $G$ has a strongly Golod-Shafarevich
subgroup $H$ of finite index. In addition, $H$ can be taken to be of
$p$th power index and subnormal, or even normal, in $G$.
\end{thm}

\begin{proof}
We prove that there exists $H\unlhd_{p^N}G$, for some $N\in\mathbb{N}$,
satisfying (\ref{eq:GS}).

Let us say that $G$ has a presentation $\langle X|R\rangle$
with $d$ generators and $r$ relators. So $\text{def}(G;X,R)=d-r$ and we set 
$\text{def}_{p}(G)-1=\varepsilon>0$ and $k$ (which in general will be
negative) to be $d-r-1$.
As in Section 2, we consider any $p$-series\[
\ldots\trianglelefteq_{p}G_{n}\trianglelefteq_{p}G_{n-1}
\trianglelefteq_{p}\ldots\trianglelefteq_{p}G_{1}\trianglelefteq_{p}G_{0}=G\]
and consider the presentation of $G_n$ given by successive Puchta
rewriting. If we were merely to use the Reidemeister-Schreier
procedure, we would have $\text{def}(G_{n})-1\ =\ p^{n}(\text{def}(G)-1)$
but Puchta rewriting always gives a deficiency at least as big, so we have
$\mbox{def}(G_n)\ge p^nk+1$ for these Puchta presentations.

From \eqref{eq:KeyD} we have 
$\text{d}_{p}(G_n)\ge\text{def}_{p}(G_n)=p^n\epsilon +1$.
When $p^n\epsilon\ge 2$ this gives us 
\[4\mbox{def}(G_n)+(\mbox{d}_p(G_n))^2-4\mbox{d}_p(G_n)
=4\mbox{def}(G_n)+(\mbox{d}_p(G_n)-1)^2-4
\ge 4+4p^nk+(p^n\epsilon -2)^2-4\]
which is quadratic in $p^n$ with leading term $\epsilon^2p^{2n}$.
Thus we have strongly Golod-Shafarevich presentations for all but
finitely many $G_n$, and from any subnormal $G_n$ we can continue
the series so that $G_N$ is normal in $G$ for some $N\ge n$. 
\end{proof}

\section{Infinite quotients which have (T) or are amenable}

Amenability is an important and very robust property: it is preserved under
subgroups, quotients, extensions and directed unions. Moreover any group
containing $F_2$ cannot be amenable. The first counterexamples to the converse
were the torsion groups constructed by Adjan-Novikov and Olshanski\u{\i}.
However the Grigorchuk and Gupta-Sidki examples are amenable.

We can also ask whether there exist residually finite groups not containing
$F_2$ which are non amenable (if one exists then the directed union property
means there will be a finitely generated non amenable subgroup which will
still be residually finite). This was completely open until \cite{ers1}
by Ershov 
which showed that there exist GS groups for sufficiently large primes $p$
with property (T). These will have a $p$-quotient which is still GS (hence
infinite) and with (T), hence non amenable. Although this does not ensure
residual finiteness, the image of such a group in its pro-$p$ completion
(which is infinite) is a dense residually finite-$p$ group, thus this 
image is infinite and has (T), as (T) is preserved by quotients (as well as
extensions and finite index subgroups).

Shortly afterwards, other constructions of finitely generated residually
finite torsion groups which are non amenable were simultaneously but
independently given by Osin in \cite{osrg} and Schlage-Puchta in \cite{key-19},
and these methods are able to produce $p$-groups for all primes $p$. In fact
once we have existence of such torsion groups, we find many other groups
have these as quotients.
\begin{prop}
Every large finitely generated group $G$ has a non amenable residually
finite torsion quotient, and every $p$-large group $G$ has a
non amenable residually finite $p$-quotient.
\end{prop}

\begin{proof}
As in the start of the proof of Corollary \ref{impq}, we can take 
$K\unlhd_f G$ (and of $p$-power index in the $p$-large case) with $K/L$
equal to such a torsion group (alternatively a $p$-group). Now 
being residually finite and being torsion are preserved under all subgroups,
finite index supergroups and finite direct products, therefore by the easier
version of Proposition \ref{pmn} mentioned after the proof, we have
$N\unlhd G$ with $G/N$ being infinite, residually finite and torsion
(and if $K/L$ is a $p$-group then so is $K/N$, being a subdirect product
of $K/L$). As $G/N$ contains $K/N$ which surjects onto $K/L$, we have that
$G/N$ is non amenable (and if $K/N$ is a $p$-group then so is $G/K$,
hence $G/N$ too).
\end{proof}

In both Osin's and Schlage-Puchta's arguments
the method used to show non amenability was that
of rank gradient: if $d(G)$ is the minimum number of generators of
the finitely generated group $G$ then the {\it rank gradient} 
$\mbox{RG}(G)$
is defined to be the infimum of $(d(H)-1)/[G:H]$ over all finite index
subgroups $H$ of $G$. If $\mbox{RG}(G)>0$ and $G$ is finitely presented then it
was shown in \cite{key-23} that $G$ is non amenable, and this was extended
to finitely generated groups in \cite{abnk}.

Schlage-Puchta showed positive rank gradient as follows: on taking a
Puchta presentation giving rise to an infinite $p$-group $G$, any finite
index subgroup $H$ of $G$ must be subnormal with index a power of $p$,
because $H$ must contain $N\unlhd_f G$ and $G/N$ is a finite $p$-group.
Now $d(H)-1\ge d_p(H)-1\ge\mbox{def}_p(H)-1$ and we have
$(\mbox{def}_p(H)-1)/[G:H]=(\mbox{def}_p(G)-1)$ which is strictly positive and
independent of $H$. Although we do not know if $G$ is necessarily
residually finite, the quotient $G/R_p$, where $R_p$ is the intersection
of all normal subgroups of $p$ power index, will be residually finite and will
also have positive rank gradient, because any homomorphism from $H$ to
$(C_p)^n$ must factor through $R_p$ and $[G/R_p:H/R_p]=[G:H]$.
We can use this to obtain:
\begin{prop}\label{pnam}
All Puchta groups are non amenable.
\end{prop}
\begin{proof}
If $G$ is a group with a presentation $\langle X|R\rangle$
having $p$-deficiency strictly
greater than 1 then $G$ need not be a $p$-group, but one can add very
high $p$th powers of every word in $F(X)$ whilst still keeping the
$p$-deficiency higher than 1. Now we have a $p$-quotient of $G$ which
must have positive rank gradient and so be non amenable, meaning that
$G$ is too.
\end{proof}

To return to GS groups, in \cite{osrg} it is asked whether
torsion GS groups have strictly positive rank gradient. By taking a Puchta
presentation yielding a $p$ group and using Theorems \ref{nallp} and
\ref{rmingp}, we see that there must exist GS $p$-groups with strictly
positive rank gradient for every prime $p$.

However there exist GS groups that are not Puchta; indeed
if we take Ershov's examples of GS groups with (T) then they cannot even
be commensurable with a Puchta group by Corollary \ref{pnot}, in contrast
to the other way round as proved in the last section. 
But a major theorem of Ershov in \cite{ersit}, proved a little while
after the above results, is that
every GS group has an infinite quotient with property (T). This 
immediately established that all GS groups are non amenable, which was
open until then. It also shows that every Puchta group has an infinite 
quotient with (T) because we know from the last section
that any Puchta group either is GS or has a finite index subgroup which is GS, 
and in the
latter case a finite index subgroup having an infinite quotient with (T)
implies the whole group does, by Proposition \ref{pmn} 
(which in this setting is Proposition 4.5 in \cite{ersit}). 
So this gives an alternative proof that Puchta groups are non amenable,
although a much harder one.

We might ask whether there exist torsion groups which have an infinite
quotient with (T) and another infinite quotient that is amenable. They do
exist: for instance take an example of each, such as a torsion GS group
along with the Grigorchuk or a Gupta-Sidki group (which are themselves
infinite amenable $p$-groups) and form the direct product. (We can
even obtain residually finite examples by making the factors residually
finite.) However
here we can obtain explicit examples which are both Puchta and GS.
\begin{cor} \label{infat}
For any prime $p$ there exists a $p$-group which is both
Puchta and GS and which has both an infinite amenable quotient and an
infinite quotient with (T).
\end{cor}
\begin{proof} Take an infinite amenable $p$-group, such as the Grigorchuk
or a Gupta-Sidki group, and a finite generating set $X$ for this group.
Enumerate the words of $F(X)-\{\mbox{id}\}$ as $w_1,w_2,\ldots$ and let 
$p^{n_i}$ be the order of $w_i$ in this group. Now choose $N_i$ such that
$N_i\geq n_i$ and big enough so that the presentation 
$G=\langle X|w_1^{p^{N_1}},w_2^{p^{N_2}},\ldots\rangle$  is both Puchta
and GS.

For instance in the case of the Gupta-Sidki groups when $p\ge 3$ with the
standard generators $X=\{a,t\}$, the paper \cite{key-GS} shows that $n_i$
is at most the word length $|w_i|$, so if we order the $w_i$ in order of
word length and then lexicographically, we have $n_i\le i$ so $N_i=i$ will
work to make $G$ Puchta, as well as GS for $p>3$ whereas $N_i=i+1$ will
do for $p=3$. As for the Grigorchuk group when $p=2$,  
we can take $X$ to be the standard generating set $\{a,b,c,d\}$ and 
$N_i=i+3$ will do,
as shown in \cite{dlha} Chapter VIII and associated
references. 

Thus we have an infinite quotient with (T) by Ershov's result,
but $G$ also surjects onto the relevant
Grigorchuk or Gupta-Sidki group because each relation 
$w_i^{p^{N_i}}=\mbox{id}$ holds in this group.
\end{proof}

We note also that as the Grigorchuk and Gupta-Sidki groups are residually
finite, $G$ has an infinite, residually finite, amenable quotient. This
quotient cannot have property ($\tau$), which is implied by (T) although
not vice versa, by \cite{lubk} Example 4.3.3. 
As ($\tau$) is preserved by quotients,
$G$ does not have ($\tau$) either.

In \cite{prd} a {}``large'' property was defined to be an abstract
group property preserved under finite index subgroups and supergroups,
as well as prequotients. The standard definition of largeness that we
use here was introduced in that paper and shown to be the most
restrictive {}``large'' property for the class of finitely generated
groups.

Other examples of {}``large'' properties are not having (T) or ($\tau$),
not being amenable, but also having an infinite quotient with (T) or
($\tau$) or which is amenable, by Theorem \ref{pmn}. It seems surprising
that the finitely generated $p$-groups in Corollary \ref{infat} with
such explicit and straightforward presentations have all six of these
{}``large'' properties (as well as there being examples that are
residually finite), but because the related Burnside problem is
unresolved, we do not know if there exist finitely \emph{presented}
torsion groups having the weakest non trivial {}``large'' property:
that of being infinite.
$\ $

\end{document}